\documentclass[british, 12pt, reqno]{amsart}
\usepackage[T1]{fontenc}
\usepackage[latin9]{inputenc}
\usepackage{amsmath, amsthm, amssymb, stmaryrd}
\usepackage{geometry}
\usepackage{amsmath}
\usepackage{amsthm}
\usepackage{amssymb}
\usepackage{xcolor}

\usepackage{enumerate}

\usepackage{hyperref}

\makeatletter
\numberwithin{equation}{section}
\numberwithin{figure}{section}
\theoremstyle{plain}
\newtheorem{thm}{\protect\theoremname}[section]
\theoremstyle{plain}
\newtheorem{cor}[thm]{\protect\corollaryname}
\ifx\proof\undefined
\newenvironment{proof}[1][\protect\proofname]{\par
	\normalfont\topsep6\p@\@plus6\p@\relax
	\trivlist
	\itemindent\parindent
	\item[\hskip\labelsep\scshape #1]\ignorespaces
}{%
	\endtrivlist\@endpefalse
}
\providecommand{\proofname}{Proof}
\fi
\theoremstyle{remark}

\theoremstyle{plain}
\newtheorem{lem}[thm]{\protect\lemmaname}

\makeatother

\usepackage{babel}
\providecommand{\corollaryname}{Corollary}
\providecommand{\lemmaname}{Lemma}
\providecommand{\remarkname}{Remark}
\providecommand{\theoremname}{Theorem}

\numberwithin{equation}{section}
\numberwithin{figure}{section}
\theoremstyle{plain}
\theoremstyle{plain}
\newtheorem{remark}[thm]{Remark}

\def\1{{\textcolor{red} {1}}}
\def\d1{{\textcolor{red} {d-1}}}

\def \bN {\mathbb N}

\def \bR {\mathbb R}

\def \bZ {\mathbb Z}

\def \rd {\mathrm d}

\def \by {\mathbf y}
\def \bz {\mathbf z}

\def \fN {\mathfrak N}

\def \cA {\mathcal A}
\def \cB {\mathcal B}
\def \cC {\mathcal C}

\def \cE {\mathcal E}

\def \cG {\mathcal G}

\def \cK {\mathcal K}

\def \cM {\mathcal M}

\def \cW {\mathcal W}

\def \le {\leqslant}
\def \leq {\leqslant}
\def \ge {\geqslant}
\def \geq {\geqslant}

\def \dim {\mathrm{dim}}
\def \dimh {{\mathrm{\dim_H}}}
\def \dimf {{\mathrm{\dim_F}}}

\def \Bad {{\mathrm{Bad}}}

\def \d {{\mathrm{d}}}

\def \ds1 {\mathds{1}}

\def \alp {{\alpha}}
\def \bet {{\beta}}
\def \gam {{\gamma}}
\def \del {{\delta}}
\def \eps {{\varepsilon}}
\def \kap {{\kappa}}
\def \lam {{\lambda}}
\def \ome {{\omega}}

\begin{document}

\author{Sam Chow \and Niclas Technau}
\address{Mathematics Institute, Zeeman Building, University of Warwick, Coventry CV4 7AL, United Kingdom}
\email{sam.chow@warwick.ac.uk}
\address{Graz University of Technology, Institute of Analysis and Number Theory, Steyrergasse 30/II, 8010
Graz, Austria}
\email{ntechnau@tugraz.at}

\title{Dispersion and Littlewood's conjecture}
\subjclass[2020]{11J83 (primary); 11J70, 28A78, 42A16  (secondary)}
\keywords{Diophantine approximation, uniform distribution, lacunary sequences, Rajchman measures}

\maketitle
\begin{abstract}
Let $\varepsilon>0$.
We construct an explicit, full-measure set of 
$\alpha \in[0,1]$ such that if $\gam \in \bR$ then, for almost all 
$\bet \in[0,1]$, if $\del \in \bR$ then there are 
infinitely many integers $n\geq1$ for which
\[
n \Vert n\alpha
- \gamma \Vert \cdot \Vert n\bet - \del \Vert < 
\frac{(\log \log n)^{3 + \eps}}{\log n}.
\]
This is a significant quantitative improvement over a result of the first author and Zafeiropoulos.
We show, moreover, that the exceptional set of $\bet$ has Fourier dimension zero, alongside further applications to badly approximable numbers and to lacunary diophantine approximation.
Our method relies on 
a dispersion estimate
and the Three Distance Theorem.
\end{abstract}

\section{Introduction}

Littlewood's conjecture (circa 1930) is a central problem in diophantine approximation. It states that if $\alp, \bet \in \bR$ then
\[
\liminf_{n \to \infty} n \| n \alp \| \cdot \| n \bet \| = 0,
\]
where $\| \cdot \|$ denotes the distance to the nearest integer. The problem lies at the heart of Margulis's measure rigidity conjecture \cite{Mar2000} and, as such, is also a holy grail of homogeneous dynamics.

Gallagher \cite{Gal1962} determined the multiplicative approximation rate of a generic pair of reals. We write $\bN = \bZ_{>0}$ throughout.

\begin{thm}
[Gallagher, 1962]
Let $\psi: \bN \to (0,\infty)$ be monotonic, and denote by $W_2^\times(\psi)$ the set of 
$(\alp, \bet) \in [0,1]^2$ such that
\[
\| n \alp \| \cdot 
\| n \bet \| < \psi(n)
\]
holds for infinitely many $n \in \bN$. Then the Lebesgue measure of $W_2^\times(\psi)$ is
\[
\begin{cases}
0, &\text{if } \displaystyle \sum_{n=1}^\infty \psi(n) \log n < \infty \\
\\
1, &\text{if } \displaystyle 
\sum_{n=1}^\infty \psi(n) \log n = \infty.
\end{cases}
\]
\label{GallagherThm}
\end{thm}
\noindent
In particular, for Lebesgue almost all $(\alp, \bet) \in \bR^2$, we have
\[
\liminf_{n \to \infty} 
n (\log n)^2 (\log \log n) \| n \alp \| \cdot \| n \bet \| = 0.
\]
However, for the purpose of proving Littlewood's conjecture, we may assume that $\alp_1, \alp_2$ are badly approximable. Here we recall that $\alp \in \bR$ is \emph{badly approximable} if 
\[
\inf \{
n \| n \alp \| : n \in \bN \}
> 0.
\]
The set $\Bad$ of badly approximable numbers has Lebesgue measure zero so, in some sense, Gallagher's theorem fails to capture the essence of Littlewood's conjecture.
With this in mind, Pollington and Velani \cite{PV2000} established that a similar result holds for many badly approximable numbers. We write $\dimh(\cA)$ for the Hausdorff dimension of a Borel set $\cA \subseteq \bR$.

\begin{thm} [Pollington--Velani, 2000]
\label{PVthm}
Let $\alp \in \Bad$. Then there exists $\cG \subseteq \Bad$ with $\dimh(\cG) = 1$ such that if $\bet \in \cG$ then
\[
\liminf_{n \to \infty}
n (\log n)
\| n \alp \|
\cdot 
\| n \bet \| \le 1.
\]
\end{thm}

The assumption that $\alp \in \Bad$ was subsequently relaxed by the first author and Zafeiropoulos \cite{CZ2021}, and the result made `fully inhomogeneous'. We require some notation to state it. Given $\alpha\in[0,1]$, we consider its partial
quotients $a_{k}(\alpha)$ and its continuants $q_{k}(\alpha)$. Define
\[
\mathcal{K}:=
\{\alpha \in \bR:
\,\Lambda(\alpha)<\infty\}
\quad \mathrm{where} \quad 
\Lambda(\alpha) =\sup \left 
\{ \frac{\log q_{k}(\alpha)}{k}:\,k\geq1 \right \}.
\]
The set $\cK$ includes almost all real numbers, as $\Lambda(\alpha)$ almost surely equals the L\'evy constant:
\[
\Lambda(\alpha)=\frac{\pi^{2}}{12\log2}\approx1.187.
\]
Moreover, $\Bad \subset \mathcal{K}$, since badly approximable numbers have bounded partial quotients. 

\begin{thm} [Chow--Zafeiropoulos, 2021]
Let $\alp \in \cK$ and $\gam, \del \in \bR$. Then there exists $\cG \subseteq \Bad$ with $\dimh(\cG) = 1$ such that if $\bet \in \cG$ then
\[
\liminf_{n \to \infty}
n (\log n)
\| n \alp - \gam \|
\cdot 
\| n \bet - \del \| 
\le 1.
\]
\end{thm}

Haynes, Jensen and Kristensen \cite{HJK2014} investigated a similar problem in which $\cG$ does not depend on $\del$. Their main result was sharpened by the second author and and Zafeiropoulos \cite{TZ2020}, and then generalised by the first author and Zafeiropoulos~\cite{CZ2021}.

\begin{thm} [Chow--Zafeiropoulos, 2021]
\label{CZmain}
Let $\eps > 0$, $\alp \in \cK$ and $\gam \in \bR$. Then there exists $\cG \subseteq \Bad$ with $\dimh(\cG) = 1$ such that if $\bet \in \cG$ and $\del \in \bR$ then
\[
n \| n \alp - \gam \|
\cdot 
\| n \bet - \del \|
< \frac{(\log \log \log n)^{\eps + 1/2}}{\sqrt{\log n}}
\]
has infinitely many solutions $n \in \bN$.
\end{thm}

We present the following stronger version.

\begin{thm}
\label{cor2}
Let $\eps > 0$, $\alp \in \cK$ and $\gam \in \bR$. Then there exists $\cG \subseteq \Bad$ with $\dimh(\cG) = 1$ such that if $\bet \in \cG$ and $\del \in \bR$ then
\begin{equation}
\label{rate}
n \| n \alp - \gam \|
\cdot 
\| n \bet - \del \|
< \frac{(\log \log n)^{3 + \eps}}{\log n}
\end{equation}
has infinitely many solutions $n \in \bN$.
\end{thm}
\noindent The quantitative threshold here is close to that of Pollington and Velani's in Theorem \ref{PVthm}.

For $\del \in \bR$, define
\[
\Bad(\del) =
\{ \bet \in \bR:
\inf \{ 
n \| n \bet - \del \|:
n \in \bN \} > 0 \},
\]
and set
\[
\cB = \{ (\bet, \del) \in \bR^2: \bet \in \Bad \cap \Bad(\del) \}.
\]
The short argument in \cite[Section 5]{CZ2021}, based on Marstrand's slicing theorem \cite[Theorem 5.8]{Fal1986} and a result of Tseng \cite{Tse2009} on twisted approximation, delivers the following consequence.

\begin{cor}
\label{cor3}
Let $\eps > 0$, $\alp \in \cK$ and $\gam \in \bR$. Then there exists $\cG \subseteq \cB$ with $\dimh(\cG) = 2$ such that if $(\bet,\del) \in \cG$ then \eqref{rate}
has infinitely many solutions $n \in \bN$.
\end{cor}

\noindent The upshot is that here $\bet$ is $\del$-inhomogeneously badly approximable, which is ``natural and desirable'' for these problems, as discussed in the introduction of \cite{PVZZ2022}.

The Fourier decay of certain measures is an essential ingredient in the proof of Theorem \ref{cor2}. Before presenting further applications, we contextualise this using the notion of Fourier dimension. For Borel $\cA \subseteq \bR$, write $\cM(\cA)$ for the set of Borel probability measures $\mu$ supported on a compact subset of $\cA$. For $\mu \in \cM(\cA)$ and $\xi \in \bR$, we write
\[
\widehat \mu(\xi)
= \int_\bR e(-\xi x)
\d \mu(x)
\]
for its Fourier transform at $\xi$, where $e(y) = e^{2 \pi i y}$. If $\cA \subseteq [0,1]$ then, for $\xi \in \bZ$, this matches the definition of the $\xi^{\mathrm{th}}$ Fourier coefficient \cite[\S 3]{Mat2015}.
The \emph{Fourier dimension} of $\cA$ is 
\[
\dimf(\cA) = \sup \{
s \in \bR:
\exists \mu \in \cM(\cA)
\qquad
\widehat \mu(\xi) \ll_s 
(1+|\xi|)^{-s/2} \quad (\xi \in \bR) \}.
\]

\begin{remark}
Here $f \ll g$ means that $|f| \le C|g|$ pointwise, for some $C$, and the subscript $s$ means that $C$ is allowed to depend on $s$. We will use this, as well as other Vinogradov and Bachmann--Landau notations, liberally throughout.
\end{remark}

\begin{thm} 
\label{MainVariant}
Let $\eps > 0$, $\alp \in \cK$ and $\gam \in \bR$. Denote by $\cE$ the set of $\bet \in [0,1]$ for which there exists $\del \in \bR$ such that \eqref{rate} has only finitely many solutions $n \in \bN$.
Then $\dimf(\cE) = 0$.
\end{thm}

This follows immediately from the following assertion. 

\begin{thm}
\label{MainThm}
Let $\eps > 0$, $\alp \in \cK$ and $\gam \in \bR$. Denote by $\cE$ the set of $\bet \in [0,1]$ for which there exists $\del \in \bR$ such that \eqref{rate} has only finitely many solutions $n \in \bN$. Let $\mu \in \cM([0,1])$, and assume that
\begin{equation} \label{mudecay}
\widehat{\mu} (\xi) \ll 
(1+|\xi|)^{-\tau}
\qquad (\xi \in \bZ)
\end{equation}
for some $\tau > 0$. Then $\mu(\cE) = 0$.
\end{thm}

\noindent Applying Theorem \ref{MainThm} with $\mu$ as Lebesgue measure furnishes the following result.

\begin{cor} \label{cor1}
Let $\eps > 0$, $\alp \in \cK$ and $\gam \in \bR$. Then, for almost all $\bet \in \bR$, if $\del \in \bR$ then \eqref{rate}
has infinitely many solutions $n \in \bN$.
\end{cor}

\noindent This sharpens \cite[Theorem 1.5]{CZ2021}, which has
roughly $(\log n)^{\eps-1/2}$ 
on the right hand side.

\bigskip

In the course of our work, we also establish some results in lacunary approximation.

\begin{thm}
\label{LacunaryApproximation}
Let $n_1, n_2, \ldots$ be a lacunary sequence of positive integers, let $\eps > 0$, and let $\mu$ be as in Theorem \ref{MainThm}. Then, for $\mu$-almost all $\bet \in [0,1]$, if $\del \in \bR$ then
\begin{equation}
\label{LacunaryRate}
\| n_t \bet - \del \| 
< \frac{(\log t)^{3 + \eps}}{t}
\end{equation}
has infinitely many solutions $t \in \bN$.
\end{thm}

\noindent Applying Theorem \ref{LacunaryApproximation} with $\mu$ as Lebesgue measure delivers the following result.

\begin{cor}
Let $n_1, n_2, \ldots$ be a lacunary sequence of positive integers, and let $\eps > 0$. Then, for almost all $\bet \in \bR$, if $\del \in \bR$ then
\eqref{LacunaryRate}
has infinitely many solutions 
$t \in \bN$.
\end{cor}

\noindent We will also be able to infer the following badly approximable analogue.

\begin{thm}
\label{LacunaryBad}
Let $n_1, n_2, \ldots$ be a lacunary sequence of positive integers, and let $\eps > 0$. Then there exists $\cG \subseteq \Bad$ with $\dimh(\cG) = 1$ such that if $\bet \in \cG$ and $\del \in \bR$ then
\eqref{LacunaryRate}
has infinitely many solutions 
$t \in \bN$. 
\end{thm}

\noindent
Once more, we apply the short argument in \cite[Section 5]{CZ2021}, giving rise to the following result about inhomogeneously badly approximable numbers.

\begin{cor}
Let $n_1, n_2, \ldots$ be a lacunary sequence of positive integers, and let $\eps > 0$. Then there exists $\cG \subseteq \cB$ with $\dimh(\cG) = 2$ such that if $(\bet, \del) \in \cG$ then \eqref{LacunaryRate} has infinitely many solutions $t \in \bN$. 
\end{cor}

\bigskip

We stress that our results are all uniform in the shift $\del$, which limits the set of available techniques. More is known when $\bet$ is allowed to depend on $\del$. We state below a special case of \cite[Theorem 1.8]{CT}.

\begin{thm} [Chow--Technau, 2021+]
Let $\alp \in \bR$ be irrational and non-Liouville, and let $\gam, \del \in \bR$. Let $\psi: \bN \to (0,\infty)$ be monotonic with
\[
\sum_{n=1}^\infty \psi(n) \log n = \infty.
\]
Then, for almost all $\bet \in \bR$, there exist infinitely many $n \in \bN$ such that
\[
\| n \alp - \gam \| \cdot \| n \bet - \del \| < \psi(n).
\]
\label{PreviousThm}
\end{thm}

\noindent
For example, taking $\psi(n) = 1/(n \log n \cdot \log \log n)$ gives
\[
\liminf_{n \to \infty} n (\log n) \| n \alp
- \gam \| \cdot \| n \bet - \del \| = 0
\]
in this context. Theorem \ref{PreviousThm} was proved using the Bohr set machinery developed in \cite{Cho2018, CT1}, and settled a conjecture of Beresnevich, Haynes and Velani~\cite{BHV2020}. 

The lacunary theory is also well understood when the shift is fixed. We state a special case of \cite[Theorem 1]{PVZZ2022}.

\begin{thm} [Pollington--Velani--Zafeiropoulos--Zorin, 2022]
\label{PVZZthm}
Let $\del \in \bR$, let $n_1, n_2, \ldots$ be a lacunary sequence of positive integers, and let $\mu$ be as in Theorem \ref{MainThm}. Let $\psi: \bN \to [0,1]$, and denote by $\cW(\psi)$ the set of $\bet \in [0,1]$ such that 
\[
\| n_t \bet - \del \| \le \psi(t)
\]
has infinitely many solutions $t \in \bN$. Then
\[
\mu(\cW(\psi)) =
\begin{cases}
0, &\text{if } 
\displaystyle
\sum_{t=1}^\infty \psi(t) 
< \infty 
\\
\\
1, &\text{if } 
\displaystyle
\sum_{t=1}^\infty \psi(t) 
= \infty.
\end{cases}
\]
\end{thm}

\noindent Theorem \ref{LacunaryApproximation} is quantitatively close to the threshold in Theorem \ref{PVZZthm}, and it comes with the `shift-uniformity' attribute: the set of allowed $\bet$ does not depend on $\del$ in Theorem \ref{LacunaryApproximation}
(unlike Theorem \ref{PVZZthm}, which does not have this attribute).

\bigskip

The analogous inhomogeneous multiplicative approximation problem with uniformity in both $\gam$ and $\del$ has also been considered. Shapira \cite{Sha2011} was the first to obtain such a result.
Subsequently, Gorodnik and Vishe~\cite{GV2016} demonstrated the following quantitative refinement.

\begin{thm} [Gorodnik--Vishe, 2016]
There exists $\kap > 0$ such that, for almost all pairs $(\alp, \bet) \in \bR^2$,
\[
\liminf_{|n| \to \infty} |n| (\log \log \log \log \log |n|)^\kap
\| n \alp - \gam \| \cdot 
\| n \bet - \del \| = 0
\]
for any $\gam, \del \in \bR$.
\end{thm}

\noindent Our results are quantitatively much stronger than this. However, uniformity in both $\gam$ and $\del$ presents a further challenge that our method seems unable to overcome. 

\section{Lacunary diophantine inequalities}

Various authors counted solutions to certain lacunary diophantine equations. We state, for imminent use, a variant of 
a lemma of Rudnick and Zaharescu, see \cite[Lemma 2.2]{RZ2002}, in which the dependence on the number of variables is made explicit.

\begin{lem}
\label{RZ1}
Let $A_1 \ge A_2 \ge \dots \ge A_s$ be positive integers.
Then, for any $b \in \bZ$ and any $Y \in \bN$, there are at most 
$8^s Y^{s-1}$ 
integer vectors $\by \in [-Y,Y]^s$ such that
\[
|A_1 y_1 + \cdots + A_s y_s + b| \le A_1.
\]
\end{lem}

\begin{proof} This is a straightforward adaptation of the proof of \cite[Lemma 2.1]{RZ2002}. 
The only change is to use an explicit version of the Lipschitz principle.
The points lie in the region
\[
\Omega = \{ \by \in [-Y,Y]^s:
|y_1 + \lam_2 y_2 + \cdots
+ \lam_s y_s + c| \le 1 \},
\]
for some $\lam_2,\ldots,\lam_s \in (0,1]$ and some $c \in \bR$. By \cite[Theorem 1.1]{BW2014}, the number of integer points in $\Omega$ is at most
\[
2^s Y^{s-1}
+ \sum_{j=0}^{s-1} \binom{s}{j} (2Y)^j
\le 8^s Y^{s-1},
\]
as claimed.
\end{proof}

By adapting the proof of \cite[Lemma 2.3]{RZ2002}, we establish the following estimate for the number of solutions to a certain diophantine inequality. The result is sharp up to logarithmic factors, owing to possible diagonal solutions. Unlike \cite[Lemma 2.3]{RZ2002}, it is uniform in the number of variables. It is also more general, in that it deals with inequalities and the ranges are allowed to differ. Our argument is slightly more refined, so we present it in full.

\begin{lem}
\label{GeneralRZ}
Let $a(1), a(2), \ldots$ be a lacunary sequence of positive integers, and let $r \in (1,2]$ be such that
\[
a(n+1) > r a(n)
\qquad (n \in \bN).
\]
Let $s \in \bN$, let $Z \ge Y \ge 2$, and let $K \in [0, a(Z)/8]$. 
Assume that $4s \le Z$.
Then there exists a constant $ R=R(r)>1$ so that 
are at most
\[
s! R^s \max \{
(YZ \log Z)^{s/2},
Y^{s-1} Z (\log Z)^{s-1} \}
\]
integer solutions to
\[
|y_1 a(z_1) + \cdots 
+ y_s a(z_s)| \le K
\]
such that
\[
0 < |y_j| \le Y,
\quad
Z < z_j \le 2Z 
\qquad
(1 \le j \le s).
\]
\end{lem}

\begin{proof}
It suffices to count solutions with 
\[
z_1 \ge \cdots \ge z_s.
\]
To each solution we associate a partition $(B_1, \ldots, B_\ell)$ of the indices as follows. 
The set $B_1$ comprises those $j$ such that $z_j \ge z_1 - \frac{2 \log Z}{\log r}$. 
With $j_2 = \max B_1 + 1$, the set $B_2$ comprises those $j \ge j_2$ 
such that $z_j \ge z_{j_2} 
- \frac{2 \log Z}{\log r}$, and so on.
We fix such a partition, giving
\[
z_{j_k} <
z_{j_{k-1}} - 
\frac{2\log Z}{\log r}
\le 
z_{j_k - 1}
\qquad
(2 \le k \le \ell)
\]
and hence
\[
a(z_{j_k}) < \frac{a(z_{j_{k-1}})}{Z^2}
\qquad
(2 \le k \le \ell).
\]
Next, we count $(\by,\bz)$ corresponding to the partition $(B_1,\ldots,B_\ell)$.

We start by choosing $z_j$ for $j \in B_k$ whenever $|B_k| \ge 2$. If there are $t$ such values of $k$, then there are at most 
\[
Z^t\left(
\frac{3 \log Z}{\log r}
\right)^{s-\ell}
\]
choices here. Next, we choose $y_j$ whenever $B_k = \{ z_j \}$. There are at most $(2Y)^{\ell-t}$ choices here.

By the triangle inequality, we have 
$|B_1| \ge 2$ and
\[
|y_1 a(z_1) + \cdots + 
y_{j_2 - 1} 
a(z_{j_2 - 1})| 
< a(z_1).
\]
This has at most 
$8^{j_2 - 1} Y^{j_2-2}$ 
solutions 
$(y_1,\ldots,y_{j_2-1})$, 
by Lemma \ref{RZ1}.

If $|B_2| \ge 2$ then we have
\[
|y_{j_2} a(z_{j_2}) + \cdots 
+ y_{j_3 - 1} a(z_{j_3 - 1}) + b|
< a(z_2),
\]
for some $b \in \bZ$ that has been determined. This has 
at most 
$
8^{j_3 - j_2}
Y^{j_3 - j_2 - 1}
$
solutions $(y_{j_2},\ldots,
y_{j_3-1})$, by Lemma~\ref{RZ1}. 
Otherwise, $|B_2| = 1$ which implies
\[
|y_2 a(z_2) + b| 
< \frac{a(z_2)}{2}.
\]
If $y_2 > 0$ then
\[
\frac{b}{2y_2} <
a(z_2) 
< \frac{2b}{y_2},
\]
which has at most $5/\log r$ solutions $z_2$. 
Similarly, if $y_2 < 0$ then there are $5/\log r$ solutions $z_2$.

Repeating the argument, 
our total count for this partition is
\[
(240/\log r)^s Y^{s-t} Z^t (\log Z)^{s-\ell}.
\]
Because $\ell \ge t$, this quantity 
is at most
\[
(240/\log r)^s Z^t
(Y \log Z)^{s-t}.
\]
As $1 \le t \le s/2$, we obtain
\[
(240/\log r)^{s} 
\max \{
(YZ \log Z)^{s/2},
Y^{s-1} Z (\log Z)^{s-1} \}.
\]
Finally, there are at most $2^s$ partitions, 
and at most $s!$ permutations of the variables.
\end{proof}

\section{Dispersion}

We require a quantitative description of how dense a finite set of points is. Whereas previous authors have used discrepancy, we use dispersion.
Given a set 
$\mathcal{A} \subset [0,1)$ of $T \ge 2$
elements 
\[
0 \leq a_{1} < a_{2} <
\dots < a_{T} <1,
\]
its \emph{dispersion} 
(the largest gap in the nearest 
neighbour spacing statistic) is
\[
\mathrm{disp}(\mathcal{A}) =
\max_{t\leq T} (a_{t+1}-a_{t})
\quad\mathrm{where}\quad 
a_{T+1}=a_{1}+1,
\]
cf. \cite[Definition 1.15]{DT1997}.
If the elements in $\cA$ 
are chosen independently and uniformly
in $[0,1)$ at random, then 
$\mathrm{disp}(\mathcal{A})
\asymp
\frac{\log T}{T}$
with high probability --- 
see David and Nagaraja's monograph on order statistics \cite[Section 6.4]{DN03}. The following result asserts that this is approximately true when $\mathcal{A}$ is a segment of
a generically-dilated lacunary sequence. 

\begin{thm}
\label{thm: dispersion}
Let $n_1,n_2,\ldots$ be a lacunary sequence of positive integers, and let $r \in (1,2]$ be such that 
\[
n_{t+1} > r n_t
\qquad (t \in \bN).
\]
Let $\eps > 0$, and let $\mu, \tau$ be as in Theorem \ref{MainThm}. Then 
\begin{equation}
\label{DispersionBound}
\mathrm{disp}(\{ 
n_t \bet
\,\mathrm{mod}\,1:
T < t \le 2T\})
\ll_{\bet, r, \eps, \tau} 
\frac{(\log T)^{3 + \eps}}{T},
\end{equation}
for $\mu$-almost all $\bet \in [0,1]$.
\end{thm}

As $\eps > 0$ is arbitrary, it suffices to establish \eqref{DispersionBound} with the right hand side replaced by $D/T$, where
\[
D = (\log T)^{3 + 2 \eps}.
\]
The plan is to show that, generically, intervals of length roughly $D/T$ have the expected number of points. To this end, we use
a concentration argument and the union bound.

Let $\ome: \bR \rightarrow \bR_{\geq 0}$ 
be a bump function
supported on $[-1,1]$
such that
$\int_{-\infty}^{\infty}
\ome(x) \,\rd x = 1$. 
This necessarily has the Fourier decay property
\[
\widehat{\omega}(\xi) \ll_N 
(1 + |\xi|)^{-N} 
\]
for each $N > 0$.
Let $c \in (0,1)$, and define
\[
C_T(\bet) = C_{T,c}(\bet) =
\sum_{u \in \bZ} \:
\sum_{T < t \le 2T}
\ome \left(
\frac{n_t \bet - c + u}
{D/T} \right),
\qquad
C(\bet) = C_T(\bet) - D.
\]
Poisson summation yields
\[
C_T(\bet) = \frac{D}{T}
\sum_{\ell \in \bZ} \:
\sum_{T < t \le 2T}
\widehat \ome(D\ell/T)
e(\ell (n_t \bet - c)).
\]
Thus, as $\widehat \ome (0) = 1$, we have
\[
C(\bet) = 
\frac{D}{T}
\sum_{\ell \ne 0} \:
\sum_{T < t \le 2T}
\widehat \ome(D\ell/T)
e(\ell (n_t \bet - c)).
\]
Define 
\[
L = \frac{T}{(\log T)^{3+\eps}}
\]
and
\[
C_0(\bet) =
\frac{D}{T}
\sum_{0 < \vert \ell \vert \leq L}
\widehat{\omega}
\left(D\ell/T\right)
\sum_{T < t \le 2T}
e(\ell(n_{t} \bet - c)).
\]
Observe from the rapid decay of $\widehat \ome$ that 
\[
C_0(\bet) = C(\bet) + 
O((\log T)^{-100}).
\]

\begin{remark}
Here and in what follows, the implied constants are allowed to depend on $\omega$ without specific indication. However, the implied constants will always be independent of the shift parameter $c$.
\end{remark}

We may assume that $T \ge 100$.
Let $s \in \bN$ with 
$8s \le T$. 
We begin with a Fourier series expansion
\[
|C_0(\bet)|^{2s}
= \sum_{k \in \bZ} a_k e(k \bet),
\]
noting that the
series on the right is a (finite) sum. 
Our goal is to bound
\[
\| C_0 \|^{2s}_{L^{2s}(\mu)}
= \int_0^1 |C_0(\bet)|^{2s} \d \mu(\bet)
= \sum_{k \in \bZ} 
a_k \int_0^1 e(k \bet) \d \mu(\bet)
= \sum_{k \in \bZ} a_k \widehat \mu(-k).
\]
This will enable us to show that $C_0(\bet)$ is generically small, and we will see from there that $C_T(\bet)$ concentrates around $D$.

We compute the Fourier coefficients $a_k$ using the Fourier coefficients of
\[
C_0(\bet) = \sum_{m \in \bZ} 
b_m e(m \bet).
\]
For $m \in \bZ$, we have
\[
b_m = \frac{D}{T}
\sum_{(\ell,t)} 
\widehat{\omega}
\left(D\ell/T\right)
e(-\ell c),
\]
where the summation runs over all integer pairs $(\ell, t)$ such that 
\[
0 < \vert \ell \vert \le L,
\qquad
T < t \le 2T,
\qquad
\ell n_t = m.
\]
Note that $\vert \widehat{\omega}(\xi)\vert
\leq \int_{-\infty}^{\infty} 
\omega(x)  \,\rd x =1$, for any $\xi \in \bR$.
Therefore
\[
|a_k| \le (D/T)^{2s} N(k)
\qquad (k \in \bZ),
\]
where $N(k)$ counts $(\ell_1, \ldots, \ell_{2s}, t_1, \ldots, t_{2s}) \in \bZ^{4s}$ such that
\begin{equation}
\label{LTranges}
0 < |\ell_j| \le L,
\quad
T < t_j \le 2T
\qquad (1 \le j \le 2s)
\end{equation}
and
\[
\sum_{j \le s} 
(\ell_j n_{t_j} - \ell_{s+j} n_{t_{s+j}})
= k.
\]

Let $R = R(r) > 1$ be as in Lemma \ref{GeneralRZ}.

\begin{lem} \label{SmallIndex}
Let $K \in [0, n_T/8]$. Then
\[
\sum_{|k| \le K} |a_k| \le
\frac{(2s)! R^{2s} D^{2s}}{(\log T)^{2s+s\eps}}.
\]
\end{lem}

\begin{proof} Observe that
\[
\sum_{|k| \le K} |a_k| \le 
(D/T)^{2s} \fN,
\]
where $\fN$ counts 
$(\ell_1,\ldots, \ell_{2s}, 
t_1, \ldots, t_{2s}) \in \bZ^{4s}$ 
in the ranges \eqref{LTranges} such that
\[
\Biggl| \sum_{j \le s}
(\ell_j n_{t_j} - \ell_{s+j} n_{t_{s+j}})
\Biggr| \le K.
\]
By Lemma \ref{GeneralRZ}, we thus have
\[
\sum_{|k| \le K} |a_k|
\le (D/T)^{2s} (2s)!
R^{2s}
\frac{T^{2s}}{(\log T)^{2s+s\eps}}
= \frac{(2s)! R^{2s} D^{2s}}{(\log T)^{2s+s\eps}}.
\]
\end{proof}

Observe that
\begin{equation} \label{total}
\sum_{k \in \bZ} |a_k| \le
(D/T)^{2s}
\sum_{k \in \bZ} N(k)
\le (D/T)^{2s} (2LT)^{2s}
= (2DL)^{2s}.
\end{equation}

\begin{proof}
[Proof of Theorem \ref{thm: dispersion}]
We use Lemma \ref{SmallIndex} for the small indices, and the estimates \eqref{mudecay}, \eqref{total} for the large ones. This gives
\begin{align*}
\| C_0 \|^{2s}_{L^{2s}(\mu)}
&= \sum_{k \in \bZ} 
a_k \widehat \mu(-k) \ll
\sum_{|k| \le K} |a_k| 
+ K^{-\tau} 
\sum_{|k| > K} |a_k|
\\ & \le
\frac{(2s)! R^{2s} D^{2s}} {(\log T)^{2s+s\eps}}
+ \frac{(2DL)^{2s}}{ K^{\tau}}
\end{align*}
for any $K\in [0,n_T/8]$. If $T$ is sufficiently large in terms of $r, \tau$, then we can choose
\[
s = \frac{\log T}
{\sqrt{\log \log T}},
\qquad K = T^{2s/\tau},
\]
whence
\[
\| C_0 \|^{2s}_{L^{2s}(\mu)}
\ll \frac{(2s)! R^{2s} D^{2s}}{(\log T)^{2s+s\eps}}.
\]
Consequently
\begin{equation}\label{eq: BC mu}
\mu(\{ \bet\in [0,1]:
|C_T(\bet) - D| > 
D/2 \}) 
\ll \frac{(2s)!(2R)^{2s}}
{(\log T)^{2s+s\eps}}.
\end{equation}

Let us now choose a set $\cC = \cC(T) \subset [0,1)$ of cardinality $O(T/D)$ such that the balls of radius $D/T$ centred in $\cC$ cover $[0,1)$. 
As
\begin{align*}
\frac{(2s)! (2R)^{2s}}
{(\log T)^{2s+s \eps}} 
&= \exp \left(
2s \log s + O(s)
- (2+\eps) s \log \log T
\right) \\
&= \exp(2\log T 
\sqrt{\log \log T} 
+ o(\log T)
- (2+\eps) \log T \sqrt{\log \log T}) 
\\ &= \exp(o(\log T)
- \eps\log T \sqrt{\log \log T})
= T^{o(1) - \eps 
\sqrt{\log \log T}},
\end{align*}
we see that $T$ times the right hand side of \eqref{eq: BC mu} is summable over $T \in \bN$. 
Thus, by the first Borel--Cantelli lemma,
for $\mu$-almost all $\bet \in [0,1]$, if $T \in \bN$ 
is sufficiently large in terms of $\bet, r, \eps, \tau$ then
\[
C_{T,c}(\bet) > 0 \qquad (c \in \cC).
\]
In particular, for each $c \in \cC$ there exists $t = t(c) \in (T, 2T]$ 
such that
\[
\| n_t \bet - c\| 
\ll D/T.
\]
Therefore
\[
\mathrm{disp}(\{ 
n_t \bet
\,\mathrm{mod}\,1:
T < t \le 2T\})
\ll D/T,
\]
as claimed.
\end{proof}

\section{Completing the proofs}

In this section, we combine Theorem \ref{thm: dispersion} with existing tools to establish Theorems \ref{cor2}, \ref{MainThm}, \ref{LacunaryApproximation} and \ref{LacunaryBad}.
Let $f: \bN \to [0,1]$ be non-increasing with the doubling property $f(t) \ll f(2t)$, and assume that Theorem \ref{thm: dispersion}
holds with \eqref{DispersionBound} replaced by
\begin{equation}
\label{GeneralDispersionBound}
\mathrm{disp}(\{ 
n_t \bet
\,\mathrm{mod}\,1:
T < t \le 2T\})
= o(f(T)) 
\qquad (T \to \infty).
\end{equation}
By Theorem 
\ref{thm: dispersion}, applied with $\eps/2$ in place of $\eps$, we can take 
\[
f(t) = \frac{(\log t)^{3 + \eps}}{t}.
\]
We now show that Theorem 
\ref{LacunaryApproximation} holds, more generally, with $o(f(t))$ in place of 
$(\log t)^{3+\eps}/t$.

\begin{proof}
[Proof of Theorem 
\ref{LacunaryApproximation}]
For $\mu$-almost all $\bet \in [0,1]$, we have
\eqref{GeneralDispersionBound}.
Thus, if $T \in \bN$ then
there exists 
$t \in (T,2T]$ such that 
\[
\| n_t \bet - \del \|
= o(f(T)).
\]
In particular, there exists an increasing sequence of positive integers $t$ along which
\[
\| n_t \bet - \del \|
= o(f(t)).
\]
\end{proof}

We require the following variant of a lemma of the first author and Zafeiropoulos, which draws its power from a quantitative version 
of the Three Distance Theorem.
\begin{lem}[Chow--Zafeiropoulos, 2021]
\label{lem: 3DT Consequence}
Let $\alpha\in\mathcal{K}$ and $\gamma\in\mathbb{R}$.
Then there exists a sequence $n_1,n_2,\ldots$ of positive integers such that 
\[
8^{t}\leq n_{t}\leq 4e^{6\Lambda(\alpha)t},
\quad n_{t} \Vert n_{t}\alpha-\gamma\Vert\leq 8,
\quad
n_{t+1} > 2n_t
\qquad (t \in \bN).
\]
\end{lem}

\begin{proof}
See \cite[Lemma 2.1]{CZ2021}, and its proof (to see that $n_{t+1} > 2n_t$).
\end{proof}

We now show that Theorem \ref{MainThm} holds, more generally, with $f(\log n)$ in place of $(\log \log n)^{3+\eps} / \log n$ in \eqref{rate}.

\begin{proof}
[Proof of Theorem \ref{MainThm}]
Let $\alp \in \cK$ and $\gam \in \bR$, and let $n_1,n_2,\ldots$ be as in Lemma~\ref{lem: 3DT Consequence}. By Theorem \ref{LacunaryApproximation},
for $\mu$-almost all $\bet \in \bR$, if $\del \in \bR$ then
\[
\| n_t \bet - \del \| 
= o(f(t))
\]
along some infinite sequence of $t \in \bN$. As $\log n_t \ll_\alp t$, and $f$ is non-increasing with the doubling property, we have
$
f(t) \ll_\alp f(\log n_t),
$
whence
\[
n_t\cdot 
\| n_t \alp - \gam \|
\cdot 
\| n_t \bet - \del \|
= o(f(\log n_t))
\]
along this sequence.
\end{proof}

We now show that Theorem \ref{cor2} holds, more generally, 
with $f(\log n)$ in place of 
$(\log \log n)^{3+\eps} 
/ \log n$ in \eqref{rate}.

\begin{proof}
[Proof of Theorem
\ref{cor2}]
Let $s < 1$. Kaufman \cite{Kau1980} constructed
\[
\mu \in \cM(\Bad \cap [0,1])
\]
with the following two key properties:
\begin{enumerate}[(i)]
\item (Frostman dimension)
For any interval $I \subseteq [0,1]$, we have
\[
\mu(I) \ll_s \lam(I)^s,
\]
where $\lam$ denotes Lebesgue measure.
\item (Polynomial Fourier decay)
We have 
\[
\widehat \mu(\xi) \ll (1+|\xi|)^{-7/10^4}
\qquad (\xi \in \bR).
\]
\end{enumerate}
Applying Theorem \ref{MainThm} to this measure gives $\mu(\cE) = 0$, where $\cE$ is the set of $\bet \in [0,1]$ for which there exists $\del \in \bR$ such that 
\[
n \| n \alp - \gam \| \cdot
\| n \bet - \del \| 
< f(\log n)
\]
has only finitely many solutions $n \in \bN$. Choosing $\cG = \Bad \cap [0,1] \setminus \cE$, we have $\mu(\cG) = 1 > 0$. Now the mass distribution principle \cite[Chapter 4]{Fal2014} reveals that $\dimh(\cG) \ge s$. As $s$ can be taken arbitrarily close to $1$, we must have $\dimh(\cG) = 1$.
\end{proof}

\noindent
Theorem \ref{LacunaryBad} follows in the same way, applying Theorem \ref{LacunaryApproximation} instead of Theorem \ref{MainThm}.

\subsection*{Closing remarks}
We see from the proof of Theorem \ref{LacunaryApproximation} that if $T$ is sufficiently large then
\[
\# \left\{ t \in \bN: 
t \le T, \: \: \| n_t \bet - \del \|
< \frac{(\log t)^{3+\eps}}{t}
\right \}
\gg \log T,
\]
which is stronger than the conclusion that \eqref{LacunaryRate} has infinitely many solutions $t \in \bN$. This also leads to the stronger conclusion that if $N$ is sufficiently large then
\[
\# \left\{ n \in \bN:
n \le N, \: \:
n \| n \alp - \gam \| \cdot
\| n \bet - \del \| <
\frac{(\log \log n)^{3+\eps}}
{\log n} 
\right \}
\gg \log \log N
\]
in Theorem \ref{cor2}. Similar counting refinements of our other results may also be inferred.

In light of the aforementioned random heuristic, it is conceivable that one might obtain 
$(\log T)^{1+\eps}/T$ on the right hand side of \eqref{DispersionBound}.
This would be interesting in its own right, and would also strengthen our other results. Getting the optimal order of the dispersion will likely involve lengthy arguments, as this extremal order statistic is non-trivial to compute even with perfectly independent random variables at hand 
---
a luxury we are not granted in the diophantine setting!

\subsection*{Acknowledgements}

NT was supported by the Austrian Science Fund: project J-$4464$ N.
The authors thank Christoph Aistleitner, Victor Beresnevich and Agamemnon Zafeiropoulos for helpful discussions.

\end{document}